\documentclass{amsart}
\usepackage{amsmath,amssymb,amsthm,graphicx,setspace,verbatim,fullpage,bbm,nicefrac, hyperref}
\usepackage[usenames]{xcolor}
\usepackage{amsfonts}
\newtheorem{theorem}{Theorem}

\newtheorem{prop}{Proposition}

\theoremstyle{remark}

\theoremstyle{definition}

\newtheorem*{cheeger*}{Cheeger Inequality}
\newtheorem*{hcheeger*}{Higher-Order Cheeger Inequality}

\newcommand{\one}{\ensuremath{\mathbbm{1}}}

\newcommand{\e}[2][]{\mathbb{E}_{#1}\left[#2\right]}
\newcommand{\p}{{\mathbb{P}}}

\newcommand{\eps}{\varepsilon}

\newcommand{\C}{\ensuremath{\mathbb{C}}}

\newcommand{\R}{\ensuremath{\mathbb{R}}}

\makeatletter
\def\Ddots{\mathinner{\mkern1mu\raise\p@
\vbox{\kern7\p@\hbox{.}}\mkern2mu
\raise4\p@\hbox{.}\mkern2mu\raise7\p@\hbox{.}\mkern1mu}}
\makeatother

\newcommand{\paren}[1]{\ensuremath{\left( #1 \right)}}

\newcommand{\ceil}[1]{\ensuremath{\left\lceil #1 \right\rceil}}
\newcommand{\floor}[1]{\ensuremath{\left\lfloor #1 \right\rfloor}}

\renewcommand{\L}{\ensuremath{\mathcal{L}}}

\DeclareMathOperator{\VOL}{Vol}
\newcommand{\Vol}[1]{\ensuremath{\VOL\paren{#1}}}
\newcommand{\hgk}{h_G^{(k)}}

\newcommand{\lilOh}[1]{\ensuremath{\mathit{o}\!\paren{#1}}}

\newcommand{\ind}{\mathbf{1}}
\DeclareMathOperator{\MOD}{mod}
\renewcommand{\mod}{\ensuremath{\MOD}}

\DeclareMathOperator{\BER}{Ber}
\newcommand{\Ber}[1]{\ensuremath{\BER}\left(#1\right)}

\DeclareMathOperator{\TR}{tr}
\newcommand{\tr}[1]{\ensuremath{\TR}\left(#1\right)}

\newcommand{\bv}{\mathbf{v}}

\newcommand{\mL}{\mathcal{L}}
\newcommand{\bD}{\mathbf{D}}
\newcommand{\bDh}{\mathbf{D}^{-1/2}}
\newcommand{\bA}{\mathbf{A}}
\newcommand{\bI}{\mathbf{I}}

\begin{document}
\title{A Linear $k$-fold Cheeger inequality}
\author{Franklin Kenter}
\author{Mary Racliffe}

\begin{abstract}
Given an undirected graph $G$, the classical Cheeger constant, $h_G$, measures the optimal partition of the vertices into 2 parts with relatively few edges between them based upon the sizes of the parts. The well-known Cheeger's inequality states that $2 \lambda_1 \le h_G \le \sqrt {2 \lambda_1}$ where $\lambda_1$ is the minimum nontrivial eigenvalue of the normalized Laplacian matrix. 

Recent work has generalized the concept of the Cheeger constant when partitioning the vertices of a graph into $k > 2$ parts. While there are several approaches, recent results have shown these higher-order Cheeger constants to be tightly controlled by $\lambda_{k-1}$, the $(k-1)^{\textrm{th}}$ nontrivial eigenvalue, to within a quadratic factor.

We present a new higher-order Cheeger inequality with several new perspectives. First, we use an alternative higher-order Cheeger constant which considers an ``average case'' approach. We show this measure is related to the average of the first $k-1$ nontrivial eigenvalues of the normalized Laplacian matrix. Further, using recent techniques, our results provide linear inequalities using the $\infty$-norms of the corresponding eigenvectors. Consequently, unlike previous results, this result is relevant even when $\lambda_{k-1} \to 1$.
\end{abstract}

\maketitle

\section{Introduction}\label{S:intro}

Let $G=(V,E)$ be an undirected graph, and let $\mL = \bDh (\mathbf{I} - \bA) \bDh$ be the normalized Laplacian of $G$ with eigenvalues $0=\lambda_0 \le \lambda_1 \le \ldots \le \lambda_{n-1}$. It is a basic fact in spectral graph theory that $\lambda_{k-1} = 0$ if and only if $G$ has at least $k$ connected components. Additionally, if $\lambda_1 \approx 0$ then the vertices of $G$ can be partitioned into 2 parts, nearly disconnected from one another. This is formalized through the Cheeger constant and the Cheeger inequality.

The classical Cheeger constant is defined as
\[h_G = \inf_{S \subset V(G)} h(S), \text{ ~~where~~ } h(S) = \frac{e(S, \overline{S})}{\min\{\Vol{S}, \Vol{\overline{S}}\}},\]
where $e(S, \overline{S})$ is the number of edges between $S$ and its complement, and $\Vol{S}$ is the sum of vertex degrees in $S$. The classical Cheeger inequality relates $h_G$ to the first eigenvalue of the normalized Laplacian matrix, as follows.
\begin{cheeger*}[see for example \cite{chung1997spectral}]
Let $\lambda_1$ be the first nontrivial eigenvalue of a connected graph $G$. Then
\[\frac{\lambda_1}{2}\leq h_G\leq \sqrt{2\lambda_1}.\]
\end{cheeger*}


Recently, some strengthenings of Cheeger's inequality have appeared (see, for example, \cite{chung1993laplacians, kenter2014linear, kwok2013improved}). Moreover, several recent results have generalized to a so-called ``higher order'' Cheeger constant (see, for example \cite{lee2012multi, louis2011algorithmic, louis2012many}) by considering a partition of $V(G)$ into $k>2$ parts. While there are several different definitions of a $k^{\textrm{th}}$ order Cheeger constant, one approach is to define 
the $k$-fold cheeger constant to be
\[\hat h_G^{(k)} = \inf_{\mathcal{S}} \max_i h(S_i)\]
 where the infimum ranges over all partitions of vertices $\mathcal{S} = \{S_1, S_2, \ldots S_k\}$. In this case, we have:
\begin{hcheeger*}[Lee, Gharan and Trevisan, \cite{lee2012multi}]
Let $\lambda_{k-1}$ be the $(k-1)^{\textrm{th}}$ nontrivial eigenvalue of a connected graph $G$. Then
\[\frac{\lambda_{k-1}}{2}\leq \hat h_G^{(k)} \leq O(k^2) \sqrt{2\lambda_{k-1}}.\]
\end{hcheeger*}
This result formally demonstrates that if $G$ can be partitioned into $k$ parts which are nearly disconnected from one another, then $\lambda_k \approx 0$. Similar results for a variant of $\hat\hgk$ can be found in \cite{louis2011algorithmic}.

The Cheeger constant and associated spectral information can be used to find clusters in graphs; that is, subgraphs that are highly connected. This has been a topic of wide interest in both the mathematics and computer science literature (see, for example, \cite{girvan2002community, graham2010finding, lee2012multi, louis2012many, ng2002spectral, schaeffer2007graph, spielman2008local}, among many others).  

This article expands upon previous work on the Cheeger constant in two ways. First, we work with the following new notion of a $k$-fold Cheeger constant. For a given partition $\mathcal{S}=\{S_1, S_2, \dots, S_k\}$ of $V(G)$, define the Cheeger constant of the partition, $\hgk(\mathcal{S})$, to be
\[h_G^{(k)}(\mathcal{S}) = \frac{1}{k}\displaystyle\sum_{i\neq j} \frac{e(S_i, S_j)}{\min\{\Vol{S_i}, \Vol{S_j}\}}.\]
We then define the $k^{\textrm{th}}$ Cheeger constant of $G$ to be $h_G^{(k)}=\inf_{\mathcal{S}} h_G^{(k)}(\mathcal{S})$. Specifically, while previous work focused on generalizing the Cheeger constant using a ``worst case'' approach, we consider the alternative ``average case'' approach. That is, $\hat h^{(k)}_G$ requires all sets in a partition to have a small Cheeger ratio, whereas $h^{(k)}_G$ can be small even if a small number of the sets have a large Cheeger ratio. We here reproduce a lower bound for $\hgk$ that agrees with the standard Cheeger inequality when $k=2$. Second,  we extend upon previous work of the first author \cite{kenter2014linear} which gives a {\it linear} upper bound at the expense of using eigenvector norms. We prove:

\begin{theorem}\label{T:main}
Fix a constant $k$. Let $G$ be an undirected graph on $n$ vertices, with maximum degree $\Delta$, and suppose there exists a constant $\beta>0$ such that $\frac{\Delta}{\Vol{G}}=\lilOh{n^{-\beta}}$. Let $0=\lambda_0\leq\lambda_1\leq\dots\leq\lambda_{k-1}$ be the first $k$ eigenvalues of $\L$, with corresponding harmonic eigenvectors $x_0, x_1, \dots, x_{k-1}$, and suppose that $\lambda_{k-1}\leq1$. Let $\alpha = \max\{\|x_i\|_\infty\ |\ i=1, 2, \dots, k-1\}$, and let $\Lambda=\frac{1}{k}\sum_{i=1}^{k-1}(1-\lambda_i)$
. Then the $k$-fold Cheeger constant $\hgk$ satisfies
\[\frac{1}{2}-\frac{\Lambda}{2}\leq\hgk\leq \left[\frac{1}{2}-\frac{1}{4k} -\frac{(k-1)\Lambda}{4\Vol{G}\alpha^2}\right](1+\lilOh{1}).\]


\end{theorem}

In addition, if we do not have $\lambda_{k-1}\leq 1$, we have the following related theorem.

\begin{theorem}\label{T:nonpos}
Fix a constant $k$. Let $G$ be an undirected graph on $n$ vertices, with maximum degree $\Delta$, and suppose there exists a constant $\beta>0$ such that $\frac{\Delta}{\Vol{G}}=\lilOh{n^{-\beta}}$. Let $0=\lambda_0\leq\lambda_1\leq\dots\leq\lambda_{k-1}$ be the first $k$ eigenvalues of $\L$, with corresponding harmonic eigenvectors $x_0, x_1, \dots, x_{k-1}$. Let $\alpha = \sum_{i=1}^{k-1}\|x_i\|_\infty$, and let $\Lambda=\frac{1}{k}\sum_{i=1}^{k-1}(1-\lambda_i)$
. 
Then the $k$-fold Cheeger constant $\hgk$ satisfies
\[\frac{1}{2}-\frac{\Lambda}{2}\leq\hgk\leq \left[\frac{1}{2}-\frac{1}{4k} -\frac{(k-1)\Lambda}{4\Vol{G}\alpha^2}\right](1+\lilOh{1}).\]

\end{theorem}

These results have several interesting features. First, the upper bounds of previous higher-order Cheeger inequalities are generally not applicable when $\lambda_k \gg 1/k^2$ (consider the complete graph, for example). In contrast, under mild conditions, Theorem \ref{T:main} or \ref{T:nonpos} applies even if $\lambda_k \gg 1/k^2$. Additionally, the result demonstrates that the ``average case'' $k$-fold Cheeger constant is tightly controlled by the average of the first $k-1$ nontrivial eigenvalues whereas the previous ``worst case'' approaches tightly control $\hat h_G^{(k)}$ with $\lambda_k$. Finally, and perhaps most interesting, Theorem \ref{T:main} shows that the Cheeger ratio can be elegantly bounded to within a {\it linear} factor of the corresponding eigenvalues when the eigenvector norms are considered.

We note that although the bound in Theorem \ref{T:nonpos} appears much weaker than that of Theorem \ref{T:main}, it is in fact not necessarily weaker at all. Indeed, if $\Lambda$ is negative, that is, if the average of the first $k-1$ nontrivial eigenvalues is greater than 1, then Theorem \ref{T:nonpos} gives a stronger bound than Theorem \ref{T:main}, as the term involving $\Lambda$ will be positive in this case. Indeed, as seen in Section \ref{S:example}, the bound given in Theorem \ref{T:nonpos} is quite good for the complete graph $K_n$.


We present this article as follows. In Section \ref{S:pre}, we give relevant background and definitions. Then, we prove the lower bounds of Theorems \ref{T:main} and \ref{T:nonpos} in Section \ref{S:LB} and the upper bounds in Section \ref{S:UB}. Finally, we conclude with the example of applying our result to $K_n$ in Section \ref{S:example}.

\section{Preliminaries}\label{S:pre}

To prove the upper bound in Theorem \ref{T:main}, we shall use tools from both probability theory and graph theory. To begin, we define our graph-theoretic notation.

Given a graph $G$, define the adjacency matrix $\bA$ to be the square matrix, indexed by $V(G)$, with $\mathbf A_{u, v}=\ind_{u\sim v}$, the indicator of whether $\{u, v\}\in E(G)$. Define $\mathbf D$ to be the diagonal matrix indexed by $V(G)$ with $\mathbf D_{u, u}=\deg_G(u)$. For simplicity of notation, if the graph is understood, we write $d_u=\deg_G(u)$. The normalized Laplacian matrix, $\L$, is given by $\L=\bD^{-1/2}(\bI-\bA)\bD^{-1/2}$. By convention, if $G$ has an isolated vertex $u$, set $(\bD^{-1/2})_{u, u} = 0$. The eigenvalues of $\L$ will be written as $0=\lambda_0\leq \lambda_1\leq\dots\leq \lambda_{n-1}$.  For a matrix or vector $\mathbf{X}$, we let $\mathbf{X}^*$ denote the conjugate transpose of $\mathbf X$. 

For a subset $S\subset V(G)$, define $\Vol{S}=\sum_{v\in S}d_v$. Define $\Vol{G} = \Vol{V(G)}$. Write $\one_S$ to denote the vector indexed by $V(G)$ with $\one_S(v)=\ind_{v\in S}$, the indicator of whether $v$ is an element of $S$. For convenience of notation, we write $\one = \one_{V(G)}$. Given two subsets $S, T\subset V(G)$, define $e(S, T)$ to be the number of edges with one incident vertex in $S$ and the other incident vertex in $T$. It is a standard exercise in graph theory to verify the following:

\begin{prop} \label{edges} For $S, T$ subsets of $V(G)$, we have \[e(S, T)=(\bD^{1/2}\one_S)^*(\mathbf I-\L)(\bD^{1/2}\one_T).\]
\end{prop}

Given an eigenvalue-eigenvector pair $(\lambda, \mathbf v)$ for $\L$, define the harmonic eigenvector corresponding to $\lambda$ to be $\mathbf D^{-1/2}\mathbf v$. Harmonic eigenvectors can be a useful tool for analyzing the normalized Laplacian. More information about harmonic eigenvectors and their uses can be found in \cite{chung1997spectral}, for example. For any graph $G$, there is always an orthonormal basis of eigenvectors $\mathbf v_0, \mathbf v_1, \dots, \mathbf v_{n-1}$; we shall assume throughout that such a basis has been chosen, and the harmonic eigenvectors used have been derived from this basis, so that the harmonic eigenvector corresponding to $\lambda_i$ will be precisely $\bD^{-1/2}\mathbf v_i$. We note that $\mathbf v_0 = (\Vol{G})^{-1/2}\bD^{1/2}\one$.

In addition, we shall require the following tools from probability theory. Given a random variable $X$, we use $\e{X}$ to denote the expected value of $X$. If $A$ is an event in a probability space, we use $\ind_A$ to denote the $0-1$ indicator random variable for $A$. Given a matrix $\mathbf M$ whose entries are all random variables, $\e{\mathbf{M}}$ denotes the matrix of entry-wise expectation. We shall use the following result from \cite{kenter2014linear}:

\begin{prop}\label{expquadform}
Let $\mathbf{x}\in \C^n$ be a random vector whose entries are pairwise independent, and let $\mu=\e{\mathbf{x}}$. If $\bA$ is an $n\times n$ symmetric matrix with $\bA_{ii}=0$ for all $i$, then \[\e{\mathbf{x}^*\bA\mathbf{x}}=\mu^*\bA\mu.\]
\end{prop}
 
 Given a random variable $X$, we say $X$ has a Bernoulli distribution with parameter $p$ if $\p(X=1)=p$, and $\p(X=0)=1-p$. We write $X\sim \Ber{p}$. 

In addition, we shall make use of Chernoff bounds. Chernoff bounds are a class of concentration inequalities that consider sums of independent random variables. Often the variables considered are Bernoulli, though that may not be the case here. There are many versions of Chernoff bounds (see, for example, \cite{AlonandSpencer}); we shall use the following:

\begin{prop}\label{chernoff}
For $i=1, 2, \dots, k$, let $X_i$ be a nonnegative random variable with $X_i\leq \Delta$. Let $S=\sum X_i$, and let $\mu=\e{S}$. Then for any $\eps>0$,
\[\p(|S-\mu|>\eps\mu)\leq 2\exp\left(\frac{-\eps^2\mu}{3\Delta}\right).\]
\end{prop}

In addition, to prove the lower bound in Theorems \ref{T:main} and \ref{T:nonpos}, we shall make use of the following linear algebra theorem. This result can be derived as a corollary of the Courant-Fischer Theorem, and can be found, for example, as Corollary 4.3.18 in \cite{HornandJohnson}.

\begin{theorem}\label{courantfischer}
Let $\mathbf M$ be an $n\times n$ Hermitian matrix with eigenvalues $\lambda_0\leq \lambda_1\leq \dots\leq \lambda_{n-1}$. Fix $k\leq n$, and let $\mathcal{U}_{n, k}$ denote the set of $n\times k$ complex matrices with orthonormal columns. Then
\[\sum_{i=0}^{k-1}\lambda_i = \min_{\mathbf{U}\in \mathcal U_{n, k}}\tr{\mathbf{U}^*\mathbf{M}\mathbf{U}}.\]
\end{theorem}

\newcommand{\bvf}{\mathbf{f}}

Rephrased, this theorem states that if $\bvf_1, \bvf_2, \dots, \bvf_k$ is a collection of orthonormal vectors in $\C^n$ and $\mathbf M$ is a Hermitian matrix, then \[\displaystyle\sum_{i=1}^k \bvf_i^*\mathbf{M}\bvf_i \geq \displaystyle\sum_{i=0}^{k-1}\lambda_i.\]

\section{The Lower Bound}\label{S:LB}

In this section, we prove the lower bound from Theorems \ref{T:main} and \ref{T:nonpos}. 
Recall that  $\Lambda =\frac{1}{k}\sum_{i=1}^{k-1}(1-\lambda_i)$, the average of the largest $k$ eigenvalues of $\bI-\L$. 

\begin{theorem}\label{T:lowerbound}
Given $k\geq 2$, \[\hgk\geq \frac{1}{2} - \frac{\Lambda}{2}.\]
\end{theorem}

\begin{proof}
Let $\mathcal S = \{S_1, S_2, \dots, S_k\}$ be a partition of the vertices of $G$. For $1\leq j\leq k$, let $g_j = \mathbf D^{1/2}\frac{1}{\sqrt{\Vol{S_i}}} \one_{S_i}$. Note that as the $S_i$ are all disjoint, we have that the set $\{g_j\}$ is orthogonal. Moreover, $\|g_j\|^2 = \sum_{v\in S_i} \frac{d_v}{\Vol{S_i}}=1$, and hence the set $\{g_j\}$ is in fact orthonormal. Then by Theorem \ref{courantfischer}, we have
\[ \sum_{i=1}^k (g_i^*\L g_i)\geq \sum_{i=0}^{k-1}\lambda_i = k (1-\Lambda).\]


On the other hand, by Proposition \ref{edges}, we have
\[ g_i^*\L g_i  = g_i^*g_i - g_i^*(\mathbf I-\L)g_i = 1- \frac{e(S_i, S_i)}{\Vol{S_i}}.\]


Combining these two results yields \begin{eqnarray*}
1-\Lambda & \leq & \frac{1}{k}\left(\sum_{i=1}^k \left(1-\frac{e(S_i, S_i)}{\Vol{S_i}}\right)\right)\\
 & = &\frac{1}{k}\left( \sum_{i=1}^k\left(1-\frac{e(S_i, S_i)+e(S_i, \overline{S_i})}{\Vol{S_i}} + \sum_{j\neq i} \frac{e(S_i, S_j)}{\Vol{S_i}}\right)\right)\\
 & = & \frac{1}{k}\left( \sum_{i=1}^k\left(1-\frac{\Vol{S_i}}{\Vol{S_i}} + \sum_{j\neq i} \frac{e(S_i, S_j)}{\Vol{S_i}}\right)\right)\\
 & = & \frac{1}{k}\sum_{j\neq i} \left(\frac{e(S_i, S_j)}{\Vol{S_i}}+ \frac{e(S_i, S_j)}{\Vol{S_j}}\right)\\
 & \leq & 2\hgk,
\end{eqnarray*}
as desired.
\end{proof}

\section{The Upper Bound}\label{S:UB}

The proof of the upper bound in Theorem \ref{T:main} is modeled after the proof of a similar upper bound by the first author in \cite{kenter2014linear} when $k=2$. The strategy employed is to choose a vector randomly so that the expectation of this vector has useful algebraic properties, and apply concentration results for the expectation.

\begin{proof}[Proof of Theorem \ref{T:main}, upper bound]
Let $\mathbf{x}_1, \mathbf{x}_2, \dots, \mathbf{x}_{k-1}$ be the first $k-1$ nontrivial harmonic eigenvectors for $G$, with corresponding eigenvalues $\lambda_1, \lambda_2, \dots, \lambda_{k-1}$. As above, we shall assume that $\mathbf x_i=\mathbf D^{-1/2}\mathbf v_i$, where the set $\{\bv_0, \bv_1, \dots, \bv_n\}$ is an orthonormal basis for $\R^n$ composed of eigenvectors of $\L$. Recall that $\mathbf v_0=\frac{1}{\sqrt{\Vol{G}}}\bD^{1/2}\one$, and hence for each $i$, 
\[ \sum_{v\in V(G)} \mathbf{x}_i(v)d_v = \mathbf{x}_i^*\bD\one = \sqrt{\Vol{G}}(\bD^{-1/2}\bv_i)^*\bD^{1/2}\bv_0 = \sqrt{\Vol{G}} \bv_i^*\bv_0=0,\]
so each $\mathbf{x}_i$ is orthogonal to $D\one$.
 Let $\alpha = \max\{ \|\mathbf{x}_j\|_{\infty}\ |\ j=1, 2, \dots, k\}$.

Let $\delta>0$ be a constant that will be defined later. For all $v\in V(G)$, define a random variable $s_v$ by 
\[\begin{array}{l} s_v = j \hbox{ with probability } \frac{1-2\delta}{2(k-1)}+\frac{\mathbf{x}_j(v)}{2(k-1)\|\mathbf{x}_j\|_\infty} \hbox{ for }j=1, 2, \dots, k-1\\ s_v=k \hbox{ otherwise }\end{array}\]

Note that the random variables for different vertices are independent. For each $j=1, 2, \dots, k$, define $S_j = \{v\in V\ | \ s_v=j\}$. Thus we can view the random variables $s_v$ as partitioning the vertices of $G$ into $k$ sets. Let $\mathbf{w}_j=\one_{S_j}$, the indicator vector for the set $S_j$. As the choice of set $S_j$ for each vertex $v$ is independent of each other vertex, we have that the entries in $\mathbf{w}_j$ are pairwise independent (although any pair $\mathbf{w}_j$, $\mathbf w_\ell$ are not independent). Notice that for $j=1, 2, \dots, k-1$, we have
\begin{eqnarray*}
\e{\Vol{S_j}} & = & \sum_{v\in V} \p(v\in S_j) d_v\\
& = & \sum_{v\in V}\left(\left(\frac{1-2\delta}{2(k-1)}\right)d_v + \frac{1}{2(k-1)\|\mathbf{x}_j\|_\infty}\mathbf{x}_j(v)d_v\right)\\
& = & \frac{1-2\delta}{2(k-1)}\Vol{G}=:\mu,
\end{eqnarray*}
by orthogonality of $\mathbf x_j$ to $\mathbf D\one$. Note that $\mu$ is independent of the choice of $j$.

On the other hand, for a given $j\leq k-1$, we can view the vertices in $S_j$ as chosen by a sequence of independent Bernoulli random variables, $X_v$, where $X_v\sim \Ber{ \frac{1-2\delta}{2(k-1)}+\frac{\mathbf{x}_j(v)}{2(k-1)\|\mathbf{x}_j\|_\infty}}$. Thus, $\Vol{S_j}=\sum_{v\in V(G)} d_vX_v$, and by Proposition \ref{chernoff}, we have that for all $\eps>0$, 
\[\p(|\Vol{S_j}-\mu|>\eps\mu)\leq 2\exp\left(\frac{-\eps^2\mu}{3\Delta}\right),\]
where $\Delta$ is the maximum degree of a vertex in $G$. 

Let $A$ be the event that $(1-\eps)\mu<\Vol{S_j}<(1+\eps)\mu$ for all $1\leq j\leq k-1$. By the union bound, we have
\[\p(A^c)\leq 2(k-1)\exp\left(\frac{-\eps^2\mu}{3\Delta}\right).\] Thus for each $j$, we have


\begin{eqnarray*}
\e{e(S_j, S_j)} & = & \e{e(S_j, S_j)\ind_{A} + e(S_j, S_j)\ind_{\overline{A}}}\\
& \leq & \e{e(S_j, S_j)\ind_{A}} + \Vol{G}\left(2(k-1)\exp\left(\frac{-\eps^2\mu}{3\Delta}\right)\right)\\
&=& \e{\frac{e(S_j, S_j)}{\Vol{S_j}}\Vol{S_j}\ind_{A}}+2(k-1)\Vol{G}\exp\left(-\frac{\eps^2\mu}{3\Delta}\right)\\
& = &\e{\frac{e(S_j, S_j)+e(S_j, \overline{S_j})-e(S_j, \overline{S_j})}{\Vol{S_j}}\Vol{S_j}\ind_{A}}+2(k-1)\Vol{G}\exp\left(-\frac{\eps^2\mu}{3\Delta}\right)\\ 
& = &\e{\frac{\Vol{S_j}-e(S_j, \overline{S_j})}{\Vol{S_j}}\Vol{S_j}\ind_{A}}+2(k-1)\Vol{G}\exp\left(-\frac{\eps^2\mu}{3\Delta}\right).
\end{eqnarray*}


Using linearity of expectation, we can sum over $j$ to obtain

\begin{eqnarray*}
\e{\sum_{j=1}^{k-1} e(S_j, S_j)} & \le & \sum_{j=1}^{k-1}\left[ \e{\frac{\Vol{S_j}-e(S_j, \overline{S}_j)}{\Vol{S_j}}\Vol{S_j}\ind_{A}} + 2(k-1)\Vol{G}\exp\left(-\eps^2\mu/(3\Delta)\right)\right]\\
& \leq &  (1+\eps)\mu~ \e{\sum_{j=1}^{k-1}\left(1-\frac{e(S_j, \overline{S_j})}{\Vol{S_j}}\right) \ind_{A}} +2(k-1)^2\Vol{G}\exp\left(-\eps^2\mu/(3\Delta)\right)\\
& = & (1+\eps)\mu~ \e{\left(k-1-\sum_{j=1}^{k-1}\sum_{\substack{i=1\\ j\neq i}}^{k} \frac{e(S_j, S_i)}{\Vol{S_j}}\right) \ind_{A}}+2(k-1)^2\Vol{G}\exp\left(-\eps^2\mu/(3\Delta)\right)\\
& = & (1+\eps)\mu~\e{\left(k-1-\sum_{i\neq j}\left(\frac{e(S_j, S_i)}{\Vol{S_j}} + \frac{e(S_i, S_j)}{\Vol{S_i}} \right)+\sum_{j=1}^{k-1}\frac{e(S_j, S_k)}{\Vol{S_k}}\right) \ind_{A}}\\&&+2(k-1)^2\Vol{G}\exp\left(-\eps^2\mu/(3\Delta)\right)\\
& \leq & (1+\eps)\mu~\e{\left(k-1-\sum_{i\neq j}\frac{2 e(S_j, S_i)}{\min\{\Vol{S_i},\Vol{S_j}\}} \frac{1-\varepsilon}{1+\varepsilon} +1\right)
\ind_{A}}\\& & +2(k-1)^2\Vol{G}\exp\left(-\eps^2\mu/(3\Delta)\right)\\
&\leq & (1+\eps)\mu\left(k-2kh_G^{(k)} \frac{1-\varepsilon}{1+\varepsilon}\right)+2(k-1)^2\Vol{G}\exp\left(-\eps^2\mu/(3\Delta)\right).
\end{eqnarray*}

On the other hand, we also have $\mathbf{w}_j=\one_{S_j}$, 
\[\e{\mathbf{w}_j} = \left(\frac{1-2\delta}{2(k-1)}\right)\one + \frac{1}{2\|\mathbf{x}_j\|_\infty}\mathbf{x}_j,\]
and thus, as the entries in $\mathbf{w}_j$ are independent as noted above, by Proposition \ref{expquadform}, we have

\begin{eqnarray*}
\e{e(S_j, S_j)} & = & \e{(\bD^{1/2}\mathbf{w}_j)^*(\bI-\L)(\bD^{1/2}\mathbf{w}_j)}\\
& = & \left(\left(\frac{1-2\delta}{2(k-1)}\right)\one + \frac{1}{2\|\mathbf{x}_j\|_\infty}\mathbf{x}_j\right)^*\bD^{1/2}(\bI-\L)\bD^{1/2}\left(\left(\frac{1-2\delta}{2(k-1)}\right)\one + \frac{1}{2\|\mathbf{x}_j\|_\infty}\mathbf{x}_j\right)\\
& = & \left(\frac{1-2\delta}{2(k-1)}\right)^2\Vol{G} + \frac{1-\lambda_j}{4\|\mathbf{x}_j\|_\infty^2}.
\end{eqnarray*}

We therefore obtain
\[ \sum_{j=1}^{k-1}\left(\left(\frac{1-2\delta}{2(k-1)}\right)^2\Vol{G} + \frac{1-\lambda_j}{4\|\mathbf{x}_j\|_\infty^2}\right)\leq  (1+\eps)\mu\left(k-2k h_G^{(k) }\frac{1-\varepsilon}{1+\varepsilon}\right)+2(k-1)^2\Vol{G}\exp\left(-\eps^2\mu/(3\Delta)\right).\]

Recall that $\frac{\Delta}{\Vol{G}}=\lilOh{n^{-\beta}}$ by hypothesis; choose $\delta=\eps=n^{-\beta/3}$. Then we have \[\exp\left(-\eps^2\mu/(3\Delta)\right) = \exp\left(-\frac{n^{-2\beta/3}(1-2n^{-\beta/3})\Vol{G}}{6(k-1)\Delta}\right)\leq \exp(-C(1-\lilOh{1})n^{\beta/3}),\] for $C$ an appropriate constant. As $\Vol{G}\leq n^2$, we thus have that the error term satisfies
\[ 2(k-1)^2\Vol{G}\exp\left(-\eps^2\mu/(3\Delta)\right) \leq 2(k-1)^2n^2\exp(-C(1-\lilOh{1})n^{\beta/3})=\lilOh{1}.\]

Moreover, as $\|\mathbf{x}_j\|_\infty\geq \alpha$  and $1-\lambda_j\geq 0$ for all $j$, we have
\begin{equation}\label{nonpositive}\sum_{j=1}^{k-1} \frac{1-\lambda_j}{4\|\mathbf{x}_j\|_\infty^2} \geq \frac{k\Lambda}{4\alpha^2}.\end{equation} Therefore,

\begin{eqnarray*}
\sum_{j=1}^{k-1}\left( \left(\frac{1-2\delta}{2(k-1)}\right)^2\Vol{G} + \frac{1-\lambda_j}{4\|\mathbf{x}_j\|_\infty^2}\right) & \leq & (1+\eps)\mu\left(k-2 kh_G^{(k)} \frac{1-\varepsilon}{1+\varepsilon}\right)+\lilOh{1}\\
\frac{(1-2\delta)^2}{4(k-1)}\Vol{G} + \frac{k\Lambda}{4\alpha^2} & \leq &  (1+\lilOh{1})\frac{\Vol{G}}{2(k-1)}\left(k-2kh_G^{(k)}(1-\lilOh{1})\right)+\lilOh{1}
\end{eqnarray*}

Solving for $\hgk$ yields
\[
h_G^{(k)}  \leq \left[\frac{1}{2}-\frac{1}{4k} -\frac{(k-1)\Lambda}{4\Vol{G}\alpha^2}\right](1+\lilOh{1}),
\]
as desired.

\end{proof}

To obtain the proof of Theorem \ref{T:nonpos}, we note that the only inequality that fails above when $\lambda_{k-1}>1$ is (\ref{nonpositive}). To correct for this problem, we shall slightly modify the definition of the random variable $s_v$ for each $v$.

\begin{proof}[Proof of Theorem \ref{T:nonpos}, upper bound]
As in the proof of Theorem \ref{T:main}, Let $\mathbf{x}_1, \mathbf{x}_2, \dots, \mathbf{x}_{k-1}$ be the first $k-1$ nontrivial harmonic eigenvectors for $G$, with corresponding eigenvalues $\lambda_1\leq \lambda_2\leq \dots\leq \lambda_{k-1}$. Let $\alpha = \sum_{i=1}^{k-1}\|\mathbf{x}_i\|_\infty$. Let $\delta>0$, and for all $v\in V(G)$, define a random variable $s_v$ by 
\[\begin{array}{l} s_v = j \hbox{ with probability } \frac{1-2\delta}{2(k-1)}+\frac{\mathbf{x}_j(v)}{2(k-1)\alpha} \hbox{ for }j=1, 2, \dots, k-1\\ s_v=k \hbox{ otherwise }\end{array}.\]

Proceed with the proof as in Theorem \ref{T:main}, noting that we can replace inequality (\ref{nonpositive}) with the equality \[\sum_{j=1}^{k-1}\frac{1-\lambda_j}{4\alpha^2}=\frac{k\Lambda}{4\alpha^2},\]
and that all else is unchanged. The result then follows.


\end{proof}

\section{Example}\label{S:example}

As an example of an application of Theorem \ref{T:nonpos}, we consider the complete graph $K_n$. Take $k$ to be a fixed constant, and we shall consider the asymptotics of $\hgk$ as $n\to\infty$.

As with the standard Cheeger constant, it is quite clear to see that $\hgk(\mathcal{S})$ will be minimized when the $S_i$ are roughly an equipartition of $n$, that is, when there are exactly $r=n\mod k$ sets of size $\ceil{\frac{n}{k}}$ and the rest are of size $\floor{\frac{n}{k}}$. Letting $\mathcal{S}$ be such a partition, we have
\begin{eqnarray*}
h_{K_n}^{(k)} & = & \frac{1}{k}\sum_{i\neq j} \frac{|S_i||S_j|}{\min\{(n-1)|S_i|, (n-1)|S_j|\}}\\
&=&\frac{1}{k(n-1)}\left( {r\choose 2}\ceil{\frac{n}{k}}+{k-r\choose 2}\floor{\frac{n}{k}} + \left({k\choose 2}-  {r\choose 2} -{k-r\choose 2}\right) \ceil{\frac{n}{k}} \right)\\
&\sim & \frac{1}{k(n-1)}\left({k\choose2}\frac{n}{k}\right)\\
& \sim & \frac{1}{2}-\frac{1}{2k}.
\end{eqnarray*}

Recall that the Laplacian eigenvalues of $K_n$ are $\lambda_0=0$ and $\lambda_i=\frac{n}{n-1}$ for all $i>0$. Therefore, the lower bound given in Theorem \ref{T:main} yields $h_{K_n}^{(k)}\geq \frac{1}{2k}(k-1)\frac{n}{n-1}\sim \frac{1}{2}-\frac{1}{2k}$, a true estimate for the Cheeger ratio.

Moreover, $\Lambda = \frac{k-1}{k}(-\frac{1}{n-1})$. Also, we have $\bv_0$, the eigenvector corresponding to 0, is given by $\bv_0=\frac{1}{\sqrt{n}}\one$. Thus any vector perpendicular to $\bv_0$ is an eigenvector for $\frac{n}{n-1}$. Note that as $\Lambda$ is negative here, we wish to maximize $\alpha^2$ in order to minimize the upper bound. Thus we take $\bv_i=\frac{1}{\sqrt{2}} (\mathbf e_{2i-1}-\mathbf e_{2i})$. It is clear that these are orthonormal. In addition, this implies that the harmonic eigenvectors are $\mathbf x_i=\bD^{-1/2}\mathbf v_i = \frac{1}{\sqrt{2(n-1)}} (\mathbf e_{2i-1}-\mathbf e_{2i})$, and hence $\alpha^2 = \left(\frac{k-1}{\sqrt{2(n-1)}}\right)^2 = \frac{(k-1)^2}{2(n-1)}$. 

Therefore, we obtain as the upper bound from Theorem \ref{T:nonpos}

\begin{eqnarray*}
h_{K_n}^{(k)}& \leq & \left[\frac{1}{2}-\frac{1}{4k} -\frac{(k-1)\Lambda}{4\Vol{G}\alpha^2}\right](1+\lilOh{1})\\
& = & \left[\frac{1}{2}-\frac{1}{4k} -\frac{(k-1)\frac{k-1}{k}(-\frac{1}{n-1})}{4n(n-1)\frac{(k-1)^2}{2(n-1)}}\right](1+\lilOh{1})\\
& = & \left(\frac{1}{2}-\frac{1}{4k}\right)(1+\lilOh{1}),
\end{eqnarray*}
compared to the true constant $\frac{1}{2}-\frac{1}{2k}$.


\bibliographystyle{siam}
\bibliography{bib_items}

\end{document}